\documentclass{amsart}

\usepackage[letterpaper,margin=1.25in]{geometry}

\newcommand{\dftxt}[1]{\textbf{\boldmath{#1}}}

\usepackage{amssymb}

\usepackage[all]{xy}
\SilentMatrices
\SelectTips{cm}{}
\newdir{ >}{{}*!/-5pt/@{>}}
\newdir{ (}{{}*!/-5pt/@^{(}}

\usepackage{url}

\newtheorem{de}{Definition}[section]

\newtheorem{lem}[de]{Lemma}
\newtheorem{prop}[de]{Proposition}
\newtheorem{cor}[de]{Corollary}
\newtheorem{thm}[de]{Theorem}

\theoremstyle{remark}
\newtheorem{rem}[de]{Remark}
\newtheorem{ex}[de]{Example}

\newcommand{\N}{\ensuremath{\mathbb{N}}}
\newcommand{\Z}{\ensuremath{\mathbb{Z}}}

\newcommand{\R}{\ensuremath{\mathbb{R}}}
\newcommand{\C}{\ensuremath{\mathbb{C}}}

\DeclareMathOperator*{\colim}{colim}
\DeclareMathOperator{\supp}{supp}
\DeclareMathOperator{\im}{Im}
\DeclareMathOperator{\Ext}{Ext}

\newcommand{\Diff}{{\mathfrak{D}\mathrm{iff}}}
\newcommand{\Fr}{{\mathfrak{F}\mathrm{r}}}
\newcommand{\Vect}{{\mathfrak{V}\mathrm{ect}}}
\newcommand{\Top}{{\mathfrak{T}\mathrm{op}}}
\newcommand{\DVect}{{\mathfrak{D}\mathrm{Vect}}}
\newcommand{\TVect}{{\mathfrak{T}\mathrm{Vect}}}

\newcommand{\Set}{{\mathfrak{S}\mathrm{et}}}
\newcommand{\DCh}{{\mathfrak{D}\mathrm{Ch}}}
\newcommand{\hDCh}{{\mathfrak{h}\mathrm{DCh}}}

\newcommand{\blank}{-}
\newcommand{\ra}{\to}

\usepackage{ifpdf}
\ifpdf  \usepackage[pdftex,bookmarks=false]{hyperref}
\else   \usepackage[hypertex]{hyperref}
\fi

\title{Homological Algebra for Diffeological Vector Spaces}

\author{Enxin Wu}
\email{enxin.wu@univie.ac.at}
\address{Faculty of Mathematics, University of Vienna, Vienna, Austria}
\date{\textbf{Draft: \today}}

\date{\today}

\begin{document}


\begin{abstract}
Diffeological spaces are natural generalizations of smooth manifolds, introduced by J.M.~Souriau and his mathematical group in the 1980's. Diffeological vector spaces (especially fine diffeological vector spaces) were first used by P. Iglesias-Zemmour to model some infinite dimensional spaces in~\cite{I1,I2}. K.~Costello and O.~Gwilliam developed homological algebra for differentiable diffeological vector spaces in Appendix A of their book~\cite{CG}. In this paper, we present homological algebra of general diffeological vector spaces via the projective objects with respect to all linear subductions, together with some applications in analysis.
\end{abstract}

\maketitle

\tableofcontents

\section{Introduction}

The concept of a diffeological space and the terminology was formulated by J.M. Souriau and his mathematical group in the 1980's (\cite{So1,So2}) as follows:

\begin{de}
A \dftxt{diffeological space} is a set $X$
together with a specified set $\mathcal{D}_X$ of functions $U \ra X$ (called \dftxt{plots})
for every open set $U$ in $\R^n$ and for each $n \in \N$,
such that for all open subsets $U \subseteq \R^n$ and $V \subseteq \R^m$:
\begin{enumerate}
\item (Covering) Every constant map $U \ra X$ is a plot;
\item (Smooth compatibility) If $U \ra X$ is a plot and $V \ra U$ is smooth,
then the composition $V \ra U \ra X$ is also a plot;
\item (Sheaf condition) If $U=\cup_iU_i$ is an open cover
and $U \ra X$ is a set map such that each restriction $U_i \ra X$ is a plot,
then $U \ra X$ is a plot.
\end{enumerate}
We usually use the underlying set $X$ to represent the diffeological space $(X,\mathcal{D}_X)$.
\end{de}

Comparing with the concept of a smooth manifold, a diffeological space starts with a set instead of a topological space, and it uses all open subsets of Euclidean spaces for characterizing smoothness, subject to the above three axioms. It is a sheaf over a certain site and has an underlying set. This makes it a concrete sheaf; see~\cite{BH}.

The theory of diffeological spaces was further developed by several mathematicians, especially Souriau's students P. Donato and P. Iglesias-Zemmour. Recently, P. Iglesias-Zemmour published the book~\cite{I2}. We refer the reader unfamiliar with diffeological spaces to~\cite{I2} for terminology and details. Let us mention a few basic properties:

A \dftxt{smooth map} between diffeological spaces is a function sending each plot of the domain to a plot of the codomain. Diffeological spaces with smooth maps form a category, denoted by $\Diff$. The category $\Diff$ contains the category of smooth manifolds and smooth maps as a full subcategory, that is, every smooth manifold is automatically a diffeological space, and smooth maps between smooth manifolds in this new sense are the same as smooth maps between smooth manifolds in the usual sense. Moreover, the category $\Diff$ is complete, cocomplete and (locally) cartesian closed. Every (co)limit in $\Diff$ has the corresponding (co)limit in $\Set$ as the underlying set. For any diffeological spaces $X$ and $Y$, we write $C^\infty(X,Y)$ for the set of all smooth maps $X \ra Y$. There is a natural diffeology (called the \dftxt{functional diffeology}) on $C^\infty(X,Y)$ making it a diffeological space. Also, every diffeological space has a natural topology called the \dftxt{$D$-topology}, which is the same as the usual topology for every smooth manifold. See~\cite[Chapters~1 and~2]{I2} for more details.

\medskip

Vector spaces are fundamental objects in mathematics. There are corresponding objects in diffeology called diffeological vector spaces, that is, they are both diffeological spaces and vector spaces such that addition and scalar multiplication maps are both smooth. P. Iglesias-Zemmour focused on fine diffeological vector spaces and used them to model some infinite dimensional spaces in~\cite{I1,I2}. K.~Costello and O. Gwilliam developed homological algebra for differentiable diffeological vector spaces (that is, diffeological vector spaces with a presheaf of flat connections) in Appendix A of their book~\cite{CG}. In this paper, we study homological algebra of general diffeological vector spaces. Here are the main results: (1) The category of diffeological vector spaces and smooth linear maps is complete and cocomplete (Theorems~\ref{thm:dvscomplete} and~\ref{thm:dvscocomplete}). (2) For every diffeological space, there is a free diffeological vector space generated by it, together with a universal property (Proposition~\ref{prop:free}). Every diffeological vector space is a quotient vector space of a free diffeological vector space (Corollary~\ref{cor:quotient}), but \dftxt{not} every diffeological vector space is a free diffeological vector space generated by a diffeological space (Example~\ref{ex:notfree}). (3) We define tensor products and duals for diffeological vector spaces in Section~\ref{s:catprop}. (4) We define short exact sequences of diffeological vector spaces (Definition~\ref{def:ses}), and we have necessary and sufficient conditions for when a short exact sequence of diffeological vector spaces splits (Theorem~\ref{thm:splitses}). Unlike the case of vector spaces, \dftxt{not} every short exact sequence of diffeological vector spaces splits (Example~\ref{ex:nonssplit}). We also get a generalized version of Borel's theorem (Remark~\ref{rem:gBorel}). (5) Fine diffeological vector spaces (Definition~\ref{def:finedvs}) behave like vector spaces, except for taking infinite products and duals (Examples~\ref{ex:notfine} and~\ref{ex:dual}). Every fine diffeological vector space is a free diffeological vector space generated by a discrete diffeological space (in the list of facts about fine diffeological vector spaces in Section~\ref{s:finedvs}). A free diffeological vector space generated by a diffeological space is fine if and only if this diffeological space is discrete (Theorem~\ref{thm:freeinverse}). (6) Every fine diffeological vector space is Fr\"olicher (Proposition~\ref{prop:Frolicher}). (7) We define projective diffeological vector spaces as projective objects in the category of diffeological vector spaces with respect to linear subductions (Definition~\ref{def:projective}), and we have many equivalent characterizations (Remark~\ref{rem:projective}, Proposition~\ref{prop:perserveses}, Corollaries~\ref{cor:projective} and~\ref{cor:ext1=>projective}). Every fine diffeological vector space and every free diffeological vector space generated by a smooth manifold is projective (Corollaries~\ref{Cor:fine=>proj} and~\ref{Cor:freeonmfd=>proj}). But \dftxt{not} every (free) diffeological vector space is projective (Example~\ref{ex:notprojective}). However, there are enough projectives in the category of diffeological vector spaces (Theorem~\ref{thm:enoughprojectives}). (8) We have many interesting (non)-examples of (fine or projective) diffeological vector spaces (Examples~\ref{ex:notfine}, \ref{ex:dual}, \ref{ex:notfine2}, \ref{ex:projective1}, \ref{ex:projective2}, \ref{ex:notprojective}, Proposition~\ref{Prop:projective}, Corollaries~\ref{Cor:fine=>proj} and~\ref{Cor:freeonmfd=>proj}, and Remark~\ref{rem:cofibrant=>projective}). (9) We establish homological algebra for general diffeological vector spaces in Section~\ref{s:ha}. As usual, every diffeological vector space has a diffeological projective resolution, and any two diffeological projective resolutions are diffeologically chain homotopic (see the paragraph before Lemma~\ref{lem:Schanuel}). Schanuel's lemma, Horseshoe lemma and (Short) Five lemma still hold in this setting (Lemmas~\ref{lem:Schanuel}, \ref{lem:Horseshoe} and ~\ref{lem:sfl}, and Proposition~\ref{prop:fl}). For any diffeological vector spaces $V$ and $W$, we can define ext diffeological vector spaces $\Ext^n(V,W)$ (Definition~\ref{def:ext}). For any short exact sequence of diffeological vector spaces in the first or second variable, we get a sequence of ext diffeological vector spaces which is exact in the category $\Vect$ of vector spaces and linear maps (Theorems~\ref{thm:les} and~\ref{thm:les2}). And $\Ext^1(V,W)$ classifies all short exact sequences in $\DVect$ of the form $0 \ra W \ra A \ra V \ra 0$ up to equivalence (Theorem~\ref{thm:bijection}).

Finally, I'd like to thank J.D.~Christensen for helpful discussions of some material in this paper, and G.~Sinnamon for the analysis proof of Example~\ref{ex:notfine2}.

\bigskip

\textbf{Convention}: 
Throughout this paper, unless otherwise specified,
\begin{itemize}
\item every vector space is over the field $\R$ of real numbers;

\item every linear map is over $\R$;

\item every tensor product is over $\R$, i.e., $\otimes = \otimes_\R$;

\item every smooth manifold is assumed to be Hausdorff, second-countable, finite dimensional and without boundary;

\item every smooth manifold is equipped with the standard diffeology;

\item every subset of a diffeological space is equipped with the sub-diffeology;

\item every product of diffeological spaces is equipped with the product diffeology;

\item every quotient set of a diffeological space is equipped with the quotient diffeology;

\item every function space of diffeological spaces is equipped with the functional diffeology.
\end{itemize}

\section{Definition and basic examples}

In this section, we recall the definitions of diffeological vector spaces and smooth linear maps between them, together with some basic examples.

\begin{de}
A \dftxt{diffeological vector space} is a vector space $V$ together with a diffeology, such that the addition map $V \times V \ra V$ and the scalar multiplication map $\R \times V \ra V$ are both smooth.
\end{de}

\begin{de}
A \dftxt{smooth linear map} between two diffeological vector spaces is a function which is both smooth and linear.
\end{de}

Diffeological vector spaces with smooth linear maps form a category, denoted by $\DVect$. Given two diffeological vector spaces $V$ and $W$, we always write $L^\infty(V,W)$ for the set of all smooth linear maps $V \ra W$.

\begin{ex}
Every vector space with the indiscrete diffeology is a diffeological vector space. We write $\R_{ind}$ for the diffeological vector space with the underlying set $\R$ and the indiscrete diffeology. 
\end{ex}

\begin{ex}\label{ex:sub}
Every linear subspace of a diffeological vector space is a diffeological vector space.
\end{ex}

\begin{ex}\label{ex:product}
Every product of diffeological vector spaces is a diffeological vector space. In particular, $\prod_{\omega} \R$, the product of countably many copies of $\R$, is a diffeological vector space.
\end{ex}

\begin{ex}\label{ex:quotient}
Every quotient vector space of a diffeological vector space is a diffeological vector space.
\end{ex}

\begin{ex}\label{ex:functionspace}
Let $X$ be a diffeological space, and let $V$ be a diffeological vector space. Then $C^\infty(X,V)$ with pointwise addition and pointwise scalar multiplication is a diffeological vector space. Moreover, the map 
\[
m:C^\infty(X,\R) \times C^\infty(X,V) \ra C^\infty(X,V) \text{ with } m(f,g)(x)=f(x)g(x)
\]
is a well-defined smooth map.
\end{ex}

\begin{ex}
Every topological vector space with the continuous diffeology is a diffeological vector space. Moreover, if we write $\TVect$ for the category of topological vector spaces and continuous linear maps, then the adjoint pair (see~\cite[Proposition~3.3]{CSW})
\[
D:\Diff \rightleftharpoons \Top:C
\]
induces an adjoint pair 
\[
D:\DVect \rightleftharpoons \TVect:C.
\]
\end{ex}

\section{Categorical properties}\label{s:catprop}

In this section, we study the categorical properties of the category $\DVect$ of diffeological vector spaces and smooth linear maps. All categorical terminology is from~\cite{M}. We begin by showing that $\DVect$ is complete and cocomplete. Then we focus on the algebraic aspects of diffeological vector spaces. We construct the free diffeological vector space over an arbitrary diffeological space, and define tensor products, duals and short exact sequences of diffeological vector spaces. We also discuss some isomorphism theorems and necessary and sufficient conditions for when a short exact sequence of diffeological vector spaces splits. Unlike the case of vector spaces, we will see in next section that \dftxt{not} every short exact sequence of diffeological vector spaces splits.

\begin{thm}\label{thm:dvscomplete}
The category $\DVect$ is complete.
\end{thm}
\begin{proof}
This follows directly from Example~\ref{ex:sub} and Example~\ref{ex:product}.
\end{proof}

The underlying set of coproduct in $\Diff$ is disjoint union of the underlying sets, so coproduct in $\Diff$ cannot be coproduct in $\DVect$. On the other hand, the forgetful functor $\DVect \ra \Vect$ is a left adjoint, so if coproduct exists in $\DVect$, the underlying set must be direct sum of the underlying vector spaces. The proof of the following proposition tells us how to put a suitable diffeology on this set to make it a coproduct in $\DVect$. 

\begin{prop}\label{prop:coproduct}
The category $\DVect$ has arbitrary coproducts.
\end{prop}
\begin{proof}
Let $\{V_i\}_{i \in I}$ be a set of diffeological vector spaces. If the index set $I$ is finite, then one can see directly that coproduct of these $V_i$ exists in $\DVect$, and it is $\prod_{i \in I} V_i$ with the product diffeology. For a general index set $I$, we first construct a category $\mathcal{I}$ with objects finite subsets of $I$ and morphisms inclusion maps. It is clear that the category $\mathcal{I}$ is filtered. There is a canonical functor $G:\mathcal{I} \ra \DVect$ sending $J \subseteq J'$ to the natural map $\prod_{j \in J} V_j \ra \prod_{j \in J'} V_j$. Write $U:\DVect \ra \Diff$ for the forgetful functor. Then by the filteredness of $\mathcal{I}$, it is easy to show that $\colim(U \circ G)$ in $\Diff$ is a diffeological vector space. Hence, it is coproduct of $\{V_i\}_{i \in I}$ in $\DVect$. In other words, it is the set $V=\oplus_{i \in I}V_i$ with the final diffeology for all canonical maps $\prod_{j \in J}V_j \ra V$ for all finite subsets $J$ of the index set $I$. 
\end{proof}

\begin{thm}\label{thm:dvscocomplete}
The category $\DVect$ is cocomplete.
\end{thm}
\begin{proof}
This follows directly from Example~\ref{ex:quotient} and Proposition~\ref{prop:coproduct}.
\end{proof}

Now we discuss how to define hom-objects in $\DVect$:

Recall that given two diffeological vector spaces $V$ and $W$, we write $L^\infty(V,W)$ for the set of all smooth linear maps $V \ra W$. Since $L^\infty(V,W)$ is a linear subspace of $C^\infty(V,W)$, by Example~\ref{ex:sub} and Example~\ref{ex:functionspace}, $L^\infty(V,W)$ with the sub-diffeology of $C^\infty(V,W)$ is a diffeological vector space. From now on, $L^\infty(V,W)$ is always equipped with this diffeology (called the \dftxt{functional diffeology}) when viewed as a diffeological (vector) space. As an easy consequence, the evaluation map $V \times L^\infty(V,W) \ra W$ is smooth.

\begin{ex}
Let $V$ be a diffeological vector space. Then the map $L^\infty(\R,V) \ra V$ defined by $f \mapsto f(1)$ is an isomorphism in $\DVect$.
\end{ex}

In order to define tensor products for diffeological vector spaces, we need the following proposition, which will be very useful throughout this paper.

\begin{prop}\label{prop:free}
The forgetful functor $\DVect \ra \Diff$ has a left adjoint.
\end{prop}
\begin{proof}
Given a diffeological space $X$, write $F(X)$ for the free vector space generated by the underlying set of $X$. For $p_1:U_1 \ra X, \cdots, p_n:U_n \ra X$ plots of $X$, write 
\[
\R \times U_1 \times \cdots \times \R \times U_n \ra F(X)
\] 
for the map defined by $(r_1,u_1,\ldots,r_n,u_n) \mapsto r_1 [p_1(u_1)]+\ldots+r_n[p_n(u_n)]$. Equip $F(X)$ with the diffeology generated by all such maps for all $n \in \Z^+$. It is clear that this is the smallest diffeology on $F(X)$ such that $F(X)$ is a diffeological vector space and the canonical map $i_X:X \ra F(X)$ is smooth. We call $F(X)$ with this diffeology the \dftxt{free diffeological vector space generated by the diffeological space $X$}. Moreover, we have the universal property that for any diffeological vector space $V$ and any smooth map $f:X \ra V$, there exists a unique smooth linear map $g:F(X) \ra V$ making the following triangle commutative:
\[
\xymatrix@C5pt{X \ar[dr]_f \ar[rr]^-{i_X} && F(X)  \ar[dl]^{g} \\ & V.}
\]
Therefore, we can define $F:\Diff \ra \DVect$ by sending $X \ra Y$ to the unique smooth linear map $F(X) \ra F(Y)$ defined by the universal property. Then $F$ is a functor which is left adjoint to the forgetful functor $\DVect \ra \Diff$.
\end{proof}

In particular, if $f:X \ra Y$ is a subduction of diffeological spaces, then $F(f):F(X) \ra F(Y)$ is a linear subduction of (free) diffeological vector spaces.

\begin{rem}\label{rem:naturalisomorphisms}
The universal property of free diffeological vector spaces says that if $X$ is a diffeological space and $V$ is a diffeological vector space, then there is a natural bijection between $C^\infty(X,V)$ and $L^\infty(F(X),V)$. Indeed this is an  isomorphism in $\DVect$.

Similarly, if $\{X_i\}_{i \in I}$ is a set of diffeological spaces and $\{V_j\}_{j \in J} \cup \{V\}$ is a set of diffeological vector spaces, then we have natural isomorphisms in $\DVect$ between $C^\infty(\coprod_{i \in I} X_i,V)$ and $\prod_{i \in I} C^\infty(X_i,V)$, between $L^\infty(V,\prod_{j \in J} V_j)$ and $\prod_{j \in J} L^\infty(V,V_j)$, and between $L^\infty(\oplus_{j \in J} V_j,V)$ and $\prod_{j \in J} L^\infty(V_j,V)$. 
\end{rem}

\begin{ex}\label{ex:notfree}
\dftxt{Not} every diffeological vector space is a free diffeological vector space generated by a diffeological space. $\R_{ind}$ is such an example. 
\end{ex}

Now we discuss tensor products in $\DVect$:

Let $V$ and $W$ be two diffeological vector spaces. Since the vector space $V \otimes W$ is a quotient vector space of the free diffeological vector space $F(V \times W)$ generated by the product space $V \times W$, by Example~\ref{ex:quotient}, $V \otimes W$ is a diffeological vector space with the quotient diffeology. From now on, we always equip $V \otimes W$ with this diffeology.

As usual, we have the following adjoint pair:

\begin{thm}
For any diffeological vector space $V$, there is an adjoint pair
\[
\blank \otimes V:\DVect \rightleftharpoons \DVect:L^\infty(V,\blank).
\]
\end{thm}
\begin{proof}
This follows from the universal property of free diffeological vector space in the proof of Proposition~\ref{prop:free}, and the fact that for diffeological vector spaces $V$, $W$ and $A$, a map $W \ra L^\infty(V,A)$ is smooth linear if and only if the adjoint set map $W \times V \ra A$ is smooth bilinear.
\end{proof}

Here are some basic properties of tensor product of diffeological vector spaces.

\begin{rem}\
\begin{enumerate}
\item Given diffeological vector spaces $V_1$, $V_2$ and $V_3$, $V_1 \otimes V_2$ is naturally isomorphic to $V_2 \otimes V_1$ in $\DVect$, and $(V_1 \otimes V_2) \otimes V_3$ is naturally isomorphic to $V_1 \otimes (V_2 \otimes V_3)$ in $\DVect$. The second isomorphism follows from the fact that the canonical projection map $F(V_1 \times V_2 \times V_3)$ to either $(V_1 \otimes V_2) \otimes V_3$ or $V_1 \otimes (V_2 \otimes V_3)$ is a (linear) subduction.

\item Given a set of diffeological vector spaces $\{V_i\}_{i \in I} \cup \{W\}$, $(\oplus_{i \in I} V_i) \otimes W$ is isomorphic to $\oplus_{i \in I} (V_i \otimes W)$ in $\DVect$.

\item For any diffeological vector space $V$, the map $V \ra V \otimes \R$ defined by $v \mapsto v \otimes 1$ is an isomorphism in $\DVect$.
\end{enumerate}
\end{rem}

\begin{prop}\label{prop:productvstensor}
Let $X$ and $Y$ be diffeological spaces. Then $F(X \times Y)$ is isomorphic to $F(X) \otimes F(Y)$ in $\DVect$.
\end{prop}
\begin{proof}
Since $i_X:X \ra F(X)$ and $i_Y:Y \ra F(Y)$ are smooth, so is $i_X \times i_Y:X \times Y \ra F(X) \times F(Y)$ and hence the composite $i_{F(X) \times F(Y)} \circ (i_X \times i_Y):X \times Y \ra F(F(X) \times F(Y))$, which induces a smooth linear map $F(X \times Y) \ra F(F(X) \times F(Y))$. So we get a smooth linear map $F(X \times Y) \ra F(X) \otimes F(Y)$ given by $\sum c_j [x_j,y_j] \mapsto \sum c_j [x_j] \otimes [y_j]$. It is known from general algebra that this map is an isomorphism in $\Vect$. We are left to show that the canonical projection map $F(F(X) \times F(Y)) \ra F(X \times Y)$ given by $\sum_i a_i [\sum_j b_{ij} [x_j], \sum_k c_{ik} [y_k]] \mapsto \sum_{i,j,k} a_i b_{ij} c_{ik} [x_j,y_k]$ is smooth. This follows directly from the description of the diffeology on the free diffeological vector space generated by a diffeological space in the proof of Proposition~\ref{prop:free}.
\end{proof}

Now we discuss the dual to a diffeological vector space:

Let $V$ be a diffeological vector space. Write $D(V)$ for the dual diffeological vector space $L^\infty(V,\R)$. Then $D$ is a functor $\DVect \ra \DVect^{op}$, and we have a natural transformation $1 \ra D^2:\DVect \ra \DVect$. 

\begin{ex}
It is \dftxt{not} true that for every finite dimensional diffeological vector space $V$, the canonical map $V \ra D^2(V)$ is a diffeomorphism. And it is \dftxt{not} true that for every diffeological vector space $V$, the canonical map $V \ra D^2(V)$ is injective. For example, $D(\R_{ind})=\R^0=D^2(\R_{ind})$. On the other hand, the canonical map $V \ra D^2(V)$ is injective if and only if $D(V)$ separates points, that is, for any $v \neq v' \in V$, there exists $l \in D(V)$ such that $l(v) \neq l(v')$.
\end{ex}

Here are some isomorphism theorems:

\begin{prop}\label{prop:iso}\
\begin{enumerate}
\item Let $f:V \ra W$ be a linear subduction between diffeological vector spaces. Then $\tilde{f}:V/\ker(f) \ra W$ defined by $\tilde{f}(v+\ker(f))=f(v)$ is an isomorphism in $\DVect$.

\item If $A$ is a diffeological vector space, $B$ is a linear subspace of $A$, and $C$ is a linear subspace of $B$, then $(A/C)/(B/C)$ is isomorphic to $A/B$ in $\DVect$.
\end{enumerate}
\end{prop}
\begin{proof}
This is easy.
\end{proof}

Now the following result is clear:

\begin{cor}\label{cor:quotient}
Every diffeological vector space is isomorphic to a quotient vector space of a free diffeological vector space in $\DVect$. 
\end{cor}
\begin{proof}
Let $V$ be a diffeological vector space. Then the smooth map $1_V:V \ra V$ induces a smooth linear map $\eta:F(V) \ra V$ such that $\eta \circ i_V=1_V$, where $i_V:V \ra F(V)$ is the canonical map. This equality implies that $\eta$ is a subduction. Therefore, $V$ is isomorphic to $F(V)/\ker(\eta)$ in $\DVect$.
\end{proof}

\begin{rem}
If $V$ is a diffeological vector space, then the canonical map $i_V:V \ra F(V)$ is smooth but \dftxt{not} linear. Therefore, the short exact sequence $0 \ra \ker(\eta) \ra F(V) \ra V \ra 0$ of diffeological vector spaces in the proof of the above corollary does \dftxt{not} split smoothly in general; see Example~\ref{ex:notprojective}. When $V$ is projective (see Definition~\ref{def:projective}), the above short exact sequence splits smoothly.
\end{rem}

\begin{de}\label{def:ses}
Let $A$, $V$ and $B$ be diffeological vector spaces. Let $i:A \ra V$ and $p:V \ra B$ be smooth linear maps. We say that 
\[
\xymatrix{0 \ar[r] & A \ar[r]^i & V \ar[r]^p & B \ar[r] & 0}
\]
is a short exact sequence in $\DVect$, if 
\begin{enumerate}
\item it is a short exact sequence in $\Vect$;

\item $i$ is an induction and $p$ is a subduction.
\end{enumerate}

We call $i$ a \dftxt{linear induction} and $p$ a \dftxt{linear subduction}.
\end{de}

One can show easily that if $f:V \ra W$ is a linear subduction between diffeological vector spaces, then $f^*:L^\infty(W,A) \ra L^\infty(V,A)$ is a linear induction for any diffeological vector space $A$.

\begin{thm}\label{thm:splitses}
Let 
\[
\xymatrix{0 \ar[r] & A \ar[r]^i & V \ar[r]^p & B \ar[r] & 0}
\]
be a short exact sequence in $\DVect$. The following are equivalent.
\begin{enumerate}
\item There exists a smooth linear map $r:V \ra A$ such that $r \circ i=id_A$.

\item There exists a smooth linear map $q:B \ra V$ such that $p \circ q=id_B$.

\item The short exact sequence is isomorphic to 
\[
\xymatrix{0 \ar[r] & A \ar[r]^-{i_1} & A \times B \ar[r]^-{p_2} & B \ar[r] & 0}
\]
in $\DVect$, with identity maps on $A$ and $B$. In particular, $V$ is isomorphic to $A \times B$ in $\DVect$.
\end{enumerate}
\end{thm}

In this case, we say that the short exact sequence
\[
\xymatrix{0 \ar[r] & A \ar[r]^i & V \ar[r]^p & B \ar[r] & 0}
\]
in $\DVect$ \dftxt{splits smoothly}, and we call $A$ (and $B$) a \dftxt{smooth direct summand} of $V$. We will show in Example~\ref{ex:nonssplit} that not every short exact sequence in $\DVect$ splits smoothly, and in particular not every linear subspace or quotient vector space of a diffeological vector space is a smooth direct summand.

\begin{proof}
If we only consider the statements in $\Vect$ instead of $\DVect$, then this is a standard result from algebra; see~\cite[Theorem~IV.1.18]{H} for instance. We are left to prove the smoothness of certain maps. 

(3) $\Rightarrow$ (1) and (3) $\Rightarrow$ (2) are clear.

(1) $\Rightarrow$ (3): By~\cite[The Short Five Lemma~1.17]{H}, $(r,p):V \ra A \times B$ is an isomorphism in $\Vect$. Its inverse can be written as $i+q:A \times B \ra V$ for some linear map $q:B \ra V$. It is straightforward to check that $q \circ p=id_V - i \circ r$. Hence, $q \circ p$ is smooth. Then $p$ being a subduction implies that $q$ is smooth. So $i+q$ is smooth. Therefore, $(r,p)$ is an isomorphism in $\DVect$, which implies the isomorphism of the two short exact sequences in $\DVect$.

(2) $\Rightarrow$ (3) can be proved dually as (1) $\Rightarrow$ (3). 
\end{proof}

Now the following result is direct:

\begin{cor}\label{cor:decomposition}
Let $V$ and $W$ be diffeological vector spaces. 
\begin{enumerate}
\item Let $i:V \ra W$ and $r:W \ra V$ be smooth linear maps such that $r \circ i = id_V$. Then there exists a diffeological vector space $X$ such that $W$ is isomorphic to $V \times X$ in $\DVect$.

\item Let $p:W \ra V$ and $q:V \ra W$ be smooth linear maps such that $p \circ q = id_V$. Then there exists a diffeological vector space $X$ such that $W$ is isomorphic to $V \times X$ in $\DVect$.
\end{enumerate}
\end{cor}

\begin{rem}
It is easy to see that the category $\DVect$ is additive with kernels and cokernels. However, it is \dftxt{not} abelian, since a morphism in $\DVect$ is monic if and only if the underlying set map is injective, but not necessarily an induction. Indeed, $\DVect$ is a quasi-abelian category in the sense of~\cite[Definition~1.1.3]{Sc} with strict epimorphisms the linear subductions and strict monomorphisms the linear inductions.
\end{rem}

\begin{rem}
M. Vincent discussed tensor products and duals of diffeological vector spaces in his master thesis~\cite[Chapter~2]{V}. He also showed that a vector space with a pre-diffeology (that is, a set of functions from open subsets of Euclidean spaces to this vector space) has a smallest diffeology containing the original pre-diffeology which makes it a diffeological vector space. This is more general than the construction of free diffeological vector space generated by a diffeological space in this section. 
\end{rem}

\section{More examples}

In this section, we present two more examples of diffeological vector spaces from analysis. These examples were introduced in the framework of Fr\"olicher spaces in~\cite{KM}, and we adapt the proofs to the diffeological setting. At the end, we also get a generalized version of Borel's theorem.

\begin{ex}[Seeley's extention~\cite{Se}]
Let $D_+$ be the set of all smooth functions $f:\R^n \times \R^{>0} \ra \R$ such that $\frac{\partial^{|m|} f}{\partial (x,t)^m}$ has continuous limits as $t \ra 0$, for every $m \in \N^{n+1}$. Then $D_+$ is a linear subspace of $C^\infty(\R^n \times \R^{>0},\R)$. Moreover, there is a smooth linear map 
\[
E:D_+ \ra C^\infty(\R^n \times \R,\R)
\] 
such that $E(f)(x,t)=f(x,t)$ when $t>0$.
\end{ex}

\begin{rem}\
\begin{enumerate}
\item It is shown in~\cite{Se} that the map $E$ in the above example is continuous if both $D_+$ and $C^\infty(\R^n \times \R,\R)$ are equipped with several topologies. However, it is straightforward to see that $E$ is actually smooth in the diffeological sense. In particular, $E$ is continuous if both $D_+$ and $C^\infty(\R^n \times \R,\R)$ are equipped with the $D$-topology.

\item The inclusion map $i:\R^n \times \R^{>0} \hookrightarrow \R^n \times \R$ induces a smooth linear map 
\[
i^*:C^\infty(\R^n \times \R,\R) \ra C^\infty(\R^n \times \R^{>0},\R).
\] 
It is clear that $\im(i^*) \subseteq D_+$. By abuse of notation, we write $i^*:C^\infty(\R^n \times \R,\R) \ra D_+$. Then $i^* \circ E=id_{D_+}$. Therefore, $E$ is an induction. In particular, if we equip $C^\infty(\R^n \times \R,\R)$ with the $D$-topology, then the $D$-topology on $D_+$ is the initial topology with respect to the map $E$. Moreover, the $D$-topology on $C^\infty(\R^n \times \R,\R)$ is the weak topology~\cite[Corollary~4.10]{CSW}, which is the same as `the topology of uniform convergence of each derivative on compact subsets of $\R^{n+1}$'~\cite{Se}.

\item Let 
\[
F=\{f \in C^\infty(\R^n \times \R,\R) \mid i^*(f)=0 \in C^\infty(\R^n \times \R^{>0},\R)\}.
\]
Then $F$ is a linear subspace of $C^\infty(\R^n \times \R,\R)$. Moreover, $C^\infty(\R^n \times \R,\R)$ is isomorphic to $D_+ \times F$ in $\DVect$ guaranteed by Corollary~\ref{cor:decomposition}(1), with the isomorphism given by 
\[
C^\infty(\R^n \times \R,\R) \ra D_+ \times F, \, h \mapsto (i^*(h),h-E \circ i^*(h))
\]
and 
\[
D_+ \times F \ra C^\infty(\R^n \times \R,\R), \, (g,f) \mapsto E(g)+f.
\]
\end{enumerate}
\end{rem}

\begin{ex}\label{ex:nonssplit}
Let 
\[
\phi:C^\infty(\R,\R) \ra \prod_\omega \R \text{ with } (\phi(f))_n=f^{(n)}(0).
\]
Then $\phi$ is a smooth linear map with kernel
\[
K=\{f \in C^\infty(\R,\R) \mid f^{(n)}(0)=0 \text{ for all } n \in \N\}. 
\]
By Borel's theorem (a more general version will be proved in Claim 1 below), $\phi$ is surjective. Therefore, there is a smooth linear bijection 
\[
\bar{\phi}:C^\infty(\R,\R)/K \ra \prod_\omega \R.
\]

\textbf{Claim 1} (Generalized Borel's theorem): The map $\bar{\phi}$ is an isomorphism in $\DVect$. 

We are left to show that $\phi$ is a subduction. Let $F:U \ra \prod_\omega \R$ be a plot. For any $x \in U$, fix two open neighborhoods $V$ and $V'$ of $x$ in $U$, such that $V \subset \bar{V} \subset V'$ with $\bar{V}$ compact. Fix $h \in C^\infty(\R,\R)$ such that $h$ has compact support and $h(t)=t$ in an open neighborhood of $0$. Let 
\[
\mu_0=2 \max \{1+\|F_0(x)\| \mid x \in V\},
\] 
and let 
\[
\mu_n=2 \max \{1+\mu_{n-1}\} \cup \{1+\|\frac{D_\alpha F_n(x)}{n!}(h^n)^{(j)}(t)\| \mid x \in V, |\alpha|+j \leq n, t \in \R\}.
\] 
Then $(\mu_n)$ is an increasing sequence with $\mu_n \geq 2n+2$ for all $n \in \N$. Now we define 
\[
\tilde{G}:V \times \R \ra \R \text{ by }\tilde{G}(x,t)=\sum_{m=0}^\infty \frac{F_m(x)}{m!} (\frac{h(t \mu_m)}{\mu_m})^m.
\] 
For any $t \neq 0$, this is a finite sum of smooth functions, hence $\tilde{G}$ is smooth there. For $t$ near $0$, on the one hand, one can show by Weierstrass M-test that 
\[
\sum_{m=0}^\infty \frac{D_\alpha F_m(x)}{m!} (\frac{(h(t \mu_m)^m)^{(j)}}{\mu_m^m})
\] 
is uniformly convergent for all $x \in V$, $\alpha \in \N^{\dim(U)}$ and $j \in \N$, so $\tilde{G}$ is also smooth at $t=0$, and hence $\tilde{G}$ is smooth on $V \times \R$, that is, the adjoint map 
\[
G:V \ra C^\infty(\R,\R)
\] 
is smooth; on the other hand, 
\[
\tilde{G}(x,t) \sim \sum_{m=0}^\infty \frac{F_m(x)}{m!} t^m \text{ as } t \ra 0,
\] 
so 
\[
(\phi \circ G)_n(x)=\frac{\partial^n \tilde{G}}{\partial t^n}(x,0)=F_n(x)
\] 
for all $x \in V$ and $n \in \N$. Therefore, $\phi$ is a subduction. 

\textbf{Claim 2}: The $D$-topology on $\prod_\omega \R$ is the product topology. 

Observe that the $D$-topology contains the product topology by the definition of product diffeology. Let $A$ be a subset of $\prod_\omega \R$ such that it is not in the product topology. In other words, there exists $a=(a_0,a_1,\ldots) \in A$ such that no open neighborhood of $a$ in the product topology is contained in $A$. Let $\epsilon_i \in \R^{>0}$ for $i \in \N$, and let 
\[
U_i=\prod_{j=0}^i(a_j-\epsilon_i,a_j+\epsilon_i) \times \prod_{k \in \N \setminus \{0,1,\ldots,i\}} \R.
\] 
Clearly, $U_i$ is an open neighborhood of $a$ in the product topology, so there exists $b_i=(b_{i0},b_{i1},\ldots) \in U_i \setminus A$. We can choose $\epsilon_i$'s so that $b_{in} \ra a_n$ fast as $i \ra \infty$ for each $n \in \omega$; for definition of fast convergence, see~\cite[page~17]{KM}. By the Special Curve Lemma~\cite[page~18]{KM}, there exists a smooth map $c:\R \ra \prod_\omega \R$ such that $c(1/n)=b_n$ and $c(0)=a$. Therefore, $A$ is not $D$-open.

\textbf{Claim 3}: There is \dftxt{no} smooth linear map $q:\prod_\omega \R \ra C^\infty(\R,\R)$ such that $\phi \circ q=id_{\prod_\omega \R}$.

To prove this, it is enough to show that for every smooth linear map 
\[
q:\prod_\omega \R \ra C^\infty(\R,\R), 
\] 
$\phi \circ q$ is not injective. By~\cite[Corollary~4.10]{CSW}, the $D$-topology on $C^\infty(\R,\R)$ coincides with the weak topology, or by~\cite[Proposition~4.2]{CSW}, the $D$-topology on $C^\infty(\R,\R)$ contains the compact-open topology. Hence, 
\[
A=\{f \in C^\infty(\R,\R) \mid |f(t)|<1 \text{ for all } t \text{ with } |t| \leq 1\}
\] 
is $D$-open in $C^\infty(\R,\R)$, which implies that $q^{-1}(A)$ is $D$-open in $\prod_\omega \R$ by the smoothness of $q$. It is clear that $\vec{0}=(0,0,\ldots) \in q^{-1}(A)$ since the zero function is in $A$ and $q$ is linear, and in Claim 2 we proved that the $D$-topology on $\prod_\omega \R$ coincides with the product topology. So there exists $N \in \Z^+$ such that 
\[
\vec{0} \in B=\{x \in \prod_\omega \R \mid x_n < 1/N \text{ for all } n \leq N\} \subseteq q^{-1}(A).
\]
Therefore, for any $x \in \prod_\omega \R$ with $x_n=0$ for all $n \leq N$, for any $k \in \N$, $kx \in B$. So by the linearity of $q$
\[
k(q(x))=q(kx) \in q(B) \subseteq A,
\]
which implies $q(x)(t)=0$ for all $t$ with $|t| \leq 1$ by definition of $A$. Therefore, 
\[
(\phi \circ q(x))_n=\frac{\partial^n q(x)}{\partial t^n}(0)=0 \text{ for all } n \in \N,
\]
that is, $\phi \circ q$ is not injective.

\textbf{Conclusion}: Combining Claims 1 and 3 with Theorem~\ref{thm:splitses}, we know that 
\[
\xymatrix{0 \ar[r] & K \ar[r] & C^\infty(\R,\R) \ar[r]^-\phi & \prod_\omega \R \ar[r] & 0}
\]
is a short exact sequence in $\DVect$, but it does \dftxt{not} split smoothly. In particular, neither $K=\{f \in C^\infty(\R,\R) \mid f^{(n)}(0)=0 \text{ for all } n \in \N\}$ nor $\prod_\omega \R$ is a smooth direct summand of $C^\infty(\R,\R)$, and $C^\infty(\R,\R)$ is not isomorphic to $K \times \prod_\omega \R$ in $\DVect$.
\end{ex}

\begin{rem}
On the contrary, for any $n \in \N$, $K_n=\{f \in C^\infty(\R,\R) \mid f^{(i)}(0)=0 \text{ for all } i \leq n\}$ is a smooth direct summand of $C^\infty(\R,\R)$.
\end{rem}

\begin{rem}[Generalized Borel's theorem]\label{rem:gBorel}
Combining Claim 1 of the above example and Remark~\ref{rem:globallift}, we can restate the generalized Borel's theorem as follows: For any $U$ open subset of a Euclidean space and any set of smooth functions $\{f_i:U \ra \R\}_{i \in \N}$, there exists a smooth map $d:U \times \R \ra \R$ such that $\frac{\partial^k d}{\partial y^k}(x,0)=f_k(x)$ for all $k \in \N$ and $x \in U$.
\end{rem}

\section{Fine diffeological vector spaces}\label{s:finedvs}

In this section, we recall the definition and some basic properties of fine diffeological vector spaces. They behave like vector spaces, except for infinite products and duals. We also present some examples and non-examples of fine diffeological vector spaces. The analysis proof of Example~\ref{ex:notfine2} relates to an interesting problem in analysis. We prove that the free diffeological vector space generated by a diffeological space $X$ is fine if and only if $X$ is discrete, and that every fine diffeological vector space is Fr\"olicher.

\begin{prop}[{\cite[3.7]{I2}}]\label{prop:finediffeology}
Given any vector space $V$, there exists a smallest diffeology on $V$ making it a diffeological vector space.
\end{prop}

\begin{de}[{\cite[3.7]{I2}}]\label{def:finedvs}
The smallest diffeology in Proposition~\ref{prop:finediffeology} is called the \dftxt{fine diffeology}. A vector space with the fine diffeology is called a \dftxt{fine diffeological vector space}.
\end{de}

Here are some facts about fine diffeological vector spaces.

\begin{itemize}
\item This diffeology is generated by all (injective) linear maps $\R^n \ra V$ for all $n \in \N$; see~\cite[3.8]{I2}. (In particular, this implies that the diffeological dimension of a diffeological vector space is always greater or equal to its vector space dimension. (See~\cite[Chapter~1 and~2]{I2} and~\cite[Section~1.8]{Wu} for the definition and basic properties of the diffeological dimension of a diffeological space.) The equality does \dftxt{not} always hold. For example, the diffeological dimension of $\R_{ind}$ is $\infty$.)

\item The fine diffeology on $\R^n$ is the standard diffeology; see~\cite[Exercise~66 on page~71]{I2}. 

\item There is an equivalence between the category of fine diffeological vector spaces with smooth linear maps and the category $\Vect$ of vector spaces with linear maps; see~\cite[3.10]{I2}. 

\item The forgetful functor $\DVect \ra \Vect$ has both a left adjoint and a right adjoint. The left adjoint is given by sending a vector space to the same space with the fine diffeology, and the right adjoint is given by sending a vector space to the same space with the indiscrete diffeology. In particular, if $V$ is a fine diffeological vector space and $W$ is a diffeological vector space, then $L^\infty(V,W)=L(V,W)$, that is, a smooth linear map $V \ra W$ is the same as a linear map $V \ra W$.

\item Every linear subspace of a fine diffeological vector space is again a fine diffeological vector space.

\item Finite product of fine diffeological vector spaces is again a fine diffeological vector space. But in general this is \dftxt{not} true for infinite product; see Example~\ref{ex:notfine}.

\item Every quotient vector space of a fine diffeological vector space is again a fine diffeological vector space. Any coproduct of fine diffeological vector spaces in $\DVect$ is again a fine diffeological vector space; see Proposition~\ref{prop:coproduct}. Therefore, any colimit of fine diffeological vector spaces in $\DVect$ is again a fine diffeological vector space.

\item Every fine diffeological vector space is a free diffeological vector space generated by any basis with the discrete diffeology. The converse is also true; see Theorem~\ref{thm:freeinverse}.

\item Let $V$ be a fine diffeological vector space, and let $A$ be a basis of $V$. Then $V$ is isomorphic to the coproduct of $|A|$-copies of $\R$ in $\DVect$. 

Let $W$ be an arbitrary diffeological vector space. Then $L^\infty(V,W)$ is isomorphic in $\DVect$ to the product of $A$-copies of $W$. In particular, if $V=\R^n$ and $W$ is fine, then $L^\infty(V,W)$ is also fine.

\item Tensor product of finitely many fine diffeological vector spaces is again a fine diffeological vector space. This follows from Proposition~\ref{prop:productvstensor} and the fact that finite product of discrete diffeological spaces is again a discrete diffeological space.

\item Let $V$ be a fine diffeological vector space. Then the canonical map $V \ra D^2(V)$ is injective. Furthermore, if $V$ is also finite dimensional, then $D(V)$ is fine and this canonical map is an isomorphism in $\DVect$. But in general, $D(V)$ may \dftxt{not} be fine; see Example~\ref{ex:dual}.

\item Let $0 \ra V_1 \ra V_2 \ra V_3 \ra 0$ be a short exact sequence of diffeological vector spaces. If $V_3$ is fine, then this short exact sequence splits smoothly. If $V_2$ is fine, then both $V_1$ and $V_3$ are fine, and hence this short exact sequence splits smoothly as well.

As an immediate consequence, it is \dftxt{not} true that for every diffeological vector space $V$, there is a linear subduction from a fine diffeological vector space to $V$. 
\end{itemize}

\begin{thm}\label{thm:freeinverse}
Let $X$ be a diffeological space such that $F(X)$ is a fine diffeological vector space. Then $X$ is discrete.
\end{thm}
\begin{proof}
Assume that there exist a plot $p:U \ra X$ and a point $u_0 \in U$ such that for every open neighborhood $V$ of $u_0$ in $U$, there exists $v \in V$ such that $p(v) \neq p(u_0)$. Consider the plot $q:\R \times U \ra F(X)$ given by $(t,u) \mapsto t[p(u)]$. Since $F(X)$ is fine, $q=l \circ f$ for some $n \in \N$, some open neighborhoods $A$ of $t_0 \neq 0$ in $\R$ and $B$ of $u_0$ in $U$, a smooth function $f:A \times B \ra \R^n$ and a linear map $l:\R^n \ra F(X)$. Without loss of generality, we may assume that $l$ sends the canonical basis of $\R^n$ to the canonical basis of $F(X)$, i.e., for every $i \in \{1,2.\ldots,n\}$, $l(e_i)=[x_i]$ for some $x_i \in X$. Then the equality $q=l \circ f$ simply implies that $f$ can not be smooth. Therefore, $X$ is a discrete diffeological space.
\end{proof}

\begin{ex}\label{ex:notfine}
Let $\lambda$ be a fixed smooth bump function, that is, $\lambda \in C^\infty(\R,\R)$ such that $\supp(\lambda) \subset (0,1)$ and $\im(\lambda)=[0,1]$. Let $\Phi:\R \ra \prod_{\omega} \R$ be defined by 
\[
(\Phi(t))_n=\lambda[n(n+1)(t-\frac{1}{n+1})].
\]
Then $\Phi$ is smooth, and the image of any neighborhood of $0 \in \R$ under $\Phi$ does \dftxt{not} live in any finite dimensional linear subspace of $\prod_\omega \R$.    Therefore, $\prod_\omega \R$ is \dftxt{not} fine.

Another way to see this is, if $\prod_\omega \R$ were fine, then the short exact sequence of diffeological vector spaces in Example~\ref{ex:nonssplit} would split smoothly.

Furthermore, $\oplus_\omega \R$ is a linear subspace of $\prod_\omega \R$. But by using the function $\Phi$ it is easy to see that the sub-diffeology from $\prod_\omega \R$ is different from the fine diffeology on $\oplus_\omega \R$.

Let $i:\R \ra \prod_\omega \R$ be defined by $x \mapsto (x,0,0,\ldots)$. One can check easily that $i$ is a linear induction such that $p_0 \circ i=1_{\R}$. In other words, although $\prod_\omega \R$ is not a fine diffeological vector space, it has a smooth direct summand of a fine diffeological vector space.
\end{ex}

\begin{ex}\label{ex:dual}
The dual of a fine diffeological vector space may \dftxt{not} be fine. For example, if $V=\oplus_\omega \R$, the coproduct of countably many $\R$'s in $\DVect$, then $D(V)=\prod_\omega \R$, the product of countably many $\R$'s in $\DVect$. By Example~\ref{ex:notfine}, we know that $D(V)$ is not fine.
\end{ex}

\begin{ex}\label{ex:notfine2}
The diffeological vector space $C^\infty(\R,\R)$ is \dftxt{not} fine.

Here is an analysis proof. Assume that the diffeological vector space $C^\infty(\R,\R)$ is fine. Let $f:\R \ra C^\infty(\R,\R)$ be the plot with $f(x)(y)=e^{xy}$. Then there exist an open neighborhood $U$ of $0$ in $\R$, an integer $n \in \N$, a smooth map $g:U \ra \R^n$ and a linear map $h:\R^n \ra C^\infty(\R,\R)$ such that $f|_U=h \circ g$. In other words, for any $(x,y) \in U \times \R$, $e^{xy}=\sum_{k=1}^n g_k(x)h_k(y)$ for some smooth functions $g_1,\ldots,g_n,h_1,\ldots,h_n \in C^\infty(\R,\R)$.

Now fix $\delta \in \R^{>0}$ such that $\delta,2\delta,\ldots,n\delta \in U$. The $(n+1) \times (n+1)$ matrix $A:=[e^{ij\delta^2}]_{i,j=0}^n$ is equal to 
\[
\sum_{k=1}^n[g_k(0),g_k(\delta),\ldots,g_k(n \delta)]^T [h_k(0),h_k(\delta),\ldots,h_k(n \delta)],
\] 
that is, the sum of $n$ rank $1$ matrices, hence singular. On the other hand, $A$ is a Vandermonde matrix, so its determinant is 
\begin{equation*}
\prod_{\substack{i,j=0,\ldots,n \\ i > j}} (e^{i \delta^2}-e^{j \delta^2}) \neq 0.
\end{equation*}
The contradiction implies that the diffeological vector space $C^\infty(\R,\R)$ is not fine.

\medskip

Here is an algebraic proof. In Claim 1 of Example~\ref{ex:nonssplit}, we have shown that 
\[
\xymatrix{0 \ar[r] & K \ar[r] & C^\infty(\R,\R) \ar[r]^-\phi & \prod_\omega \R \ar[r] & 0}
\]
is a short exact sequence of diffeological vector spaces. Now we have two ways to argue that $C^\infty(\R,\R)$ is not fine. The first way is, if $C^\infty(\R,\R)$ is fine, then so is $\prod_\omega \R$, but it is not by Example~\ref{ex:notfine}. The second way is, if $C^\infty(\R,\R)$ is fine, then the above short exact sequence splits smoothly, which contradicts the conclusion of Example~\ref{ex:nonssplit}.
\end{ex}

Fr\"olicher spaces are another well-studied generalization of smooth manifolds. We refer the reader to~\cite{St} for definition of the category $\Fr$ of Fr\"olicher spaces and smooth maps. There is an adjoint pair $F:\Diff \rightleftharpoons \Fr:G$; see~\cite{St}. We say that a diffeological space $X$ is Fr\"olicher if there exists a Fr\"olicher space $Y$ such that $G(Y)=X$. 

\begin{prop}\label{prop:Frolicher}
Every fine diffeological vector space is Fr\"olicher.
\end{prop}
\begin{proof}
Let $V$ be a fine diffeological vector space. Write $\mathcal{F}=C^\infty(V,\R)$ and write $\mathcal{C}=\{c:\R \ra V \mid l \circ c \in C^\infty(\R,\R) \text{ for all } l \in \mathcal{F}\}$. Then $(\mathcal{C},V,\mathcal{F})$ is a Fr\"olicher space. We are left to show that for any open subset $U$ of $\R^n$, for any set map $f:U \ra V$, if $l \circ f \in C^\infty(U,\R)$ for all $l \in \mathcal{F}$, then $f$ is a plot of $V$. 

Note that if $A$ is a basis of $V$, then $V$ is isomorphic to $\oplus_{a \in A} \R$ in $\DVect$. Hence the projection of $V$ to any $1$-dimensional linear subspace is in $\mathcal{F}$. Therefore, it is enough to show that for any $u \in U$, there exists an open neighborhood $U'$ of $u$, such that $\im(f|_{U'})$ is in a finite dimensional linear subspace of $V$.

Now assume that for every open neighborhood $U'$ of $u$ in $U$, $\im(f|_{U'})$ is not in a finite dimensional linear subspace of $V$. Then there exists a sequence $u_n \ra u$ in $U$ such that 
\begin{equation*}
\begin{split}
\{f(u_n)\}_{n \in \N} \text{ is a linearly independent subset of } V, \, & \textrm{if $f(u)=0$;} \\
\text{or } \{f(u_n)\}_{n \in \N} \cup \{f(u)\} \text{ is a linearly independent subset of } V, \, & \textrm{if $f(u) \neq 0$}.
\end{split}
\end{equation*} 
We can extend this linearly independent subset to a basis $A$ of $V$. Then by the universal property of coproduct, $g:V \ra \R$ defined by the linear extension of the map 
\[
\begin{cases}
f(u_n) \mapsto 1 \text{ for all } n \in \N \\
\text{other elements in } A \mapsto 0
\end{cases}
\]
is in $\mathcal{F}$. But clearly $g \circ f$ is not continuous, hence not smooth, contracting the assumption that $l \circ f \in C^\infty(U,\R)$ for all $l \in \mathcal{F}$.
\end{proof}

\section{Projective diffeological vector spaces}

In this section, we introduce a large class of diffeological vector spaces called projective diffeological vector spaces. They are projective with respect to linear subductions. Fine diffeological vector spaces and the free diffeological vector spaces generated by smooth manifolds are such examples. We have several equivalent characterizations of projective diffeological vector spaces in this section and next. However, \dftxt{not} every (free) diffeological vector space is projective. But there are enough projectives in the category $\DVect$ of diffeological vector spaces.

\begin{de}\label{def:projective}
We call a diffeological vector space $V$ \dftxt{projective}, if for every linear subduction $f:W_1 \ra W_2$ and every smooth linear map $g:V \ra W_2$, there exists a smooth linear map $h:V \ra W_1$ such that $g=f \circ h$.
\end{de}

\begin{prop}\label{Prop:projective}
Let $X$ be a diffeological space. Then the free diffeological vector space $F(X)$ generated by $X$ is projective if and only if for any linear subduction $f:W_1 \ra W_2$ and any smooth map $g:X \ra W_2$, there exists a smooth map $h:X \ra W_1$ making the following diagram commutative:
\[
\xymatrix{& X \ar[d]^g \ar[dl]_h \\ W_1 \ar[r]_f & W_2.}
\]
\end{prop}
\begin{proof}
This follows directly from definition of projective diffeological vector space and the universal property of free diffeological vector space generated by a diffeological space.
\end{proof}

\begin{cor}\label{Cor:fine=>proj}
Every fine diffeological vector space is projective.
\end{cor}
\begin{proof}
This follows directly from Proposition~\ref{Prop:projective} together with the fact that every fine diffeological vector space is a free diffeological vector space generated by a discrete diffeological space.
\end{proof}

\begin{cor}\label{Cor:freeonmfd=>proj}
The free diffeological vector space generated by a smooth manifold is projective.
\end{cor}
\begin{proof}
Let $M$ be a smooth manifold. For any linear subduction $f:W_1 \ra W_2$ and any smooth map $g:M \ra W_2$, by Proposition~\ref{Prop:projective}, we only need to construct a smooth map $h:M \ra W_1$ such that $g=f \circ h$.

Since $f:W_1 \ra W_2$ is a subduction and $M$ is a smooth manifold, we can find an atlas $\{U_i\}_{i \in I}$ of $M$ such that for each $i$, $U_i$ is diffeomorphic to a bounded open subset of $\R^n$ with $n=\dim(M)$, and there exists a smooth map $h_i:U_i \ra W_1$ making the following diagram commutative:
\[
\xymatrix{U_i \ \ar[drr]_{h_i} \ar@{^{(}->}[r] & M \ar[r]^g & W_2 \\ && W_1. \ar[u]_f}
\]
Let $\{\rho_i\}_{i \in I}$ be a partition of unity subordinate to this covering $\{U_i\}_{i \in I}$ of $M$. By the sheaf condition, $\rho_i h_i:M \ra W_1$ defined by $m \mapsto \rho_i(m)h_i(m)$ is smooth for each $i$. And since $W_1$ is a diffeological vector space, $h:M \ra W_1$ defined by $h(m)=\sum_{i \in I} \rho_i(m)h_i(m)$ is smooth. It is easy to check that $f(h(m))=g(m)$ for all $m \in M$.
\end{proof}

By Theorem~\ref{thm:freeinverse}, it is clear that Corollary~\ref{Cor:freeonmfd=>proj} provides a lot of projective but not fine diffeological vector spaces.

\begin{rem}\label{rem:globallift}
The above corollary implies that for any linear subduction $W_1 \ra W_2$, every plot of $W_2$ \dftxt{globally} lifts to a plot of $W_1$. This is clearly \dftxt{not} true for general subductions.

Moreover, by~\cite[8.15]{I2}, we know that every linear subduction $\pi:W_1 \ra W_2$ is a diffeological principal bundle; see~\cite[Chapter~8]{I2}. The pullback in $\Diff$ along any plot $p:U \ra W_2$ is \dftxt{globally} trivial. And for any $(u,w) \in U \times W_1$ with $p(u)=\pi(w)$, there exists a smooth map $f:U \ra W_1$ such that $p=\pi \circ f$ and $f(u)=w$.
\end{rem}

\begin{ex}\label{ex:projective1}\
\begin{enumerate}
\item Assume that
\[ 
\xymatrix{\R^0 \ar[r]^0 \ar[d]_0 & \R \ar[d]^{i_2} \\ \R \ar[r]_{i_1} & X}
\]
is a pushout diagram in $\Diff$. Then $F(X)$ is projective.

Here is the proof. Let $f:W_1 \ra W_2$ be a linear subduction, and let $g:X \ra W_2$ be a smooth map. By Corollary~\ref{Cor:freeonmfd=>proj}, there exists smooth maps $\alpha,\bar{\beta}:\R \ra W_1$ such that $f \circ \alpha=g\circ i_1$ and $f \circ \bar{\beta}=g \circ i_2$. Since $W_1$ is a diffeological vector space, we can define $\beta:\R \ra W_1$ by $\beta(y)=\bar{\beta}(y)-\bar{\beta}(0)+\alpha(0)$. Then the map $\beta$ is smooth, and  the linearity of $f$ implies that $f \circ \beta=g \circ i_2$. Now $\alpha(0)=\beta(0)$. So we have a smooth map $h:X \ra W_1$ such that $f \circ h=g$. The result then follows from Proposition~\ref{Prop:projective}. 

\item Let $X$ be the union of the coordinate axes in $\R^2$ with the sub-diffeology. Then $F(X)$ is projective.

Here is the proof. Write $i_1,i_2:\R \ra X$ for the smooth maps defined by $i_1(x)=(x,0)$ and $i_2(y)=(0,y)$, and write $i:X \ra \R^2$ for the inclusion. Let $f:W_1 \ra W_2$ be a linear subduction, and let $g:X \ra W_2$ be a smooth map. Since $W_2$ is a diffeological vector space, we can define $G:\R^2 \ra W_2$ by $G(x,y)=g(i_1(x))+g(i_2(y))-g(i_1(0))$. Then $G$ is a smooth map, and $G \circ i=g$. By Corollary~\ref{Cor:freeonmfd=>proj}, there exists a smooth map $H:\R^2 \ra W_1$ such that $f \circ H=G$. Hence $h:=H \circ i:X \ra W_1$ is a smooth map such that $f \circ h=g$.
\end{enumerate}
\end{ex}

\begin{rem}\label{rem:cofibrant=>projective}
More generally, use the terminology in~\cite{CW}, we have the following: Let $X$ and $Y$ be diffeological spaces. If either $X$ is cofibrant, or $F(Y)$ is a projective diffeological vector space and there is a cofibration $X \ra Y$, then $F(X)$ is a projective diffeological vector space. This follows from the fact that every linear subduction is a trivial fibration.
\end{rem}

However, not every (free) diffeological vector space is projective:

\begin{ex}\label{ex:notprojective}\
\begin{enumerate}
\item \dftxt{Not} every diffeological vector space is projective. From Example~\ref{ex:nonssplit}, we know that $\prod_\omega \R$ is such an example.

\item \label{ex:irrationaltorus} More surprisingly, \dftxt{not} every free diffeological vector space is projective.

Let $T_\alpha$ be the $1$-dimensional irrational torus of slope some irrational number $\alpha$, and let $\pi:\R \ra T_\alpha$ be the quotient map. Then $F(\pi):F(\R) \ra F(T_\alpha)$ is a linear subduction. We claim that $F(T_\alpha)$ is \dftxt{not} projective.

We only need to show that there exists no smooth map $h:T_\alpha \ra F(\R)$ such that $i_{T_\alpha}=F(\pi) \circ h$, where $i_{T_\alpha}:T_\alpha \ra F(T_\alpha)$ is the canonical smooth map. Otherwise, $f:=h \circ \pi \in \C^\infty(\R,F(\R))$ and $F(\pi)(f(x))=i_{T_\alpha}(\pi(x))$ for all $x \in \R$. On the one hand, $f \in \C^\infty(\R,F(\R))$ implies that there exists a connected open neighborhood $A$ of $0$ in $\R$ together with smooth maps $\alpha_i:A \ra \R$ (viewed as scalars) and $\beta_i:A \ra \R$ (viewed as base) with $1 \leq i \leq n$ for some minimum $n \in \Z^+$, such that $f(x)=\sum_{i=1}^n \alpha_i(x) [\beta_i(x)]$ for all $x \in A$. Since $(\Z + \alpha \Z) \cap A$ is dense in $A$, the map $\beta_i$ must be constant for each $i$. On the other hand, since $\pi|_A:A \ra T_\alpha$ is surjective, $F(\pi)(f(x))=i_{T_\alpha}(\pi(x))$ for all $x \in A$ implies that some $\beta_i$ can not be constant. So we reach the contradiction.

As consequences of this example,
\begin{itemize}
\item we get an independent proof that any $1$-dimensional irrational torus $T_\alpha$ is not cofibrant (see~\cite[Example~4.27]{CW} for another proof);

\item the diffeological vector space $F(T_\alpha)$ is not fine;

\item projective diffeological vector spaces are not preserved under arbitrary colimits in $\DVect$.
\end{itemize}

By a similar argument, one can show that $F^2(T_\alpha):=F(F(T_\alpha))$, the free diffeological vector space generated by $F(T_\alpha)$ is also \dftxt{not} projective. More generally, \dftxt{none} of $F^n(T_\alpha):=F(F^{n-1}(T_\alpha))$ is projective for $n \in \Z^+$.

\item $\R_{ind}$ is not projective, since there is no smooth linear map $h:\R_{ind} \ra F(\R_{ind})$ such that $\eta \circ h=1_{\R_{ind}}$, where $\eta:F(\R_{ind}) \ra \R_{ind}$ is the canonical smooth linear map induced by $1_{\R_{ind}}:\R_{ind} \ra \R_{ind}$.
\end{enumerate}
\end{ex}

Here are some basic properties for projective diffeological vector spaces:

\begin{rem}\label{rem:projective}\
\begin{enumerate}
\item \label{projiffsplit} A diffeological vector space $V$ is projective if and only if every linear subduction $W \ra V$ splits smoothly. This follows from definition of projective diffeological vector space and the fact that linear subductions are closed under pullbacks in $\DVect$.

\item \label{directsumofprojectives} Let $\{V_i\}_{i \in I}$ be a set of diffeological vector spaces. Then each $V_i$ is projective if and only if the coproduct $\oplus_{i \in I} V_i$ in $\DVect$ is projective.

\item Projective diffeological vector spaces are closed under retracts in $\DVect$, that is, if $f:V \ra W$ and $g:W \ra V$ are smooth linear maps between diffeological vector spaces such that $g \circ f=1_V$ and $W$ is projective, then $V$ is also projective.

\item Let $0 \ra V_1 \ra V_2 \ra V_3 \ra 0$ be a short exact sequence of diffeological vector spaces. If $V_3$ is projective, then this short exact sequence splits smoothly. In particular, every projective diffeological vector space is a smooth direct summand of a free diffeological vector space, but the converse is not true in general.
\end{enumerate}
\end{rem}

\begin{ex}\label{ex:projective2}\
\begin{enumerate}
\item Let $C_n$ be the cyclic subgroup of order $n$ for the multiplicative group $S^1$. Let $C_n$ act on $\R^2$ by rotation, and write $X$ for the quotient diffeological space with the quotient map $\pi:\R^2 \ra X$. Then $F(X)$ is a projective diffeological vector space. Here is the proof. Define $f:X \ra F(\R^2)$ by $x \mapsto \sum \frac{1}{n}[\bar{x}]$, where the sum is over all $\bar{x} \in \pi^{-1}(x)$. Clearly $f$ is a smooth map such that $F(X)$ is a retract in $\DVect$ of the projective diffeological vector space $F(\R^2)$. Hence, $F(X)$ is also projective.

\item Similarly, let $X_n$ be the quotient diffeological space $\R^n/\{\pm 1\}$. Then $F(X_n)$ is a projective diffeological vector space.
\end{enumerate}
\end{ex}

Recall from Example~\ref{ex:notprojective}(\ref{ex:irrationaltorus}) that the domain of the canonical linear subduction $F^2(T_\alpha) \ra F(T_\alpha)$ is not projective. So the proof of Corollary~\ref{cor:quotient} does \dftxt{not} provide us a functorial way to find a projective diffeological vector space $V$ together with a linear subduction $V \ra F(T_\alpha)$. However, there is a linear subduction $F(\pi):F(\R) \ra F(T_\alpha)$ whose domain is a projective diffeological vector space.

\begin{thm}\label{thm:enoughprojectives}
The category $\DVect$ has enough projectives, that is, for any diffeological vector space $V$, there exists a projective diffeological vector space $P(V)$ together with a linear subduction $P(V) \ra V$.
\end{thm}
\begin{proof}
Let $V$ be an arbitrary diffeological vector space. We construct $P(V)$ as the coproduct in $\DVect$ of all $F(U)$ indexed by all plots $U \ra V$. By the universal property of free diffeological vector spaces, there is a canonical smooth linear map $F(U) \ra V$. Therefore, there is a smooth linear map $P(V) \ra V$. By construction, it is easy to see that this map is a subduction. By Corollary~\ref{Cor:freeonmfd=>proj} and Remark~\ref{rem:projective}(\ref{directsumofprojectives}), we know that $P(V)$ is a projective diffeological vector space.
\end{proof}

Note that since the free functor $F:\Diff \ra \DVect$ is a left adjoint, $P(V)$ constructed in the proof of the above theorem is actually a free diffeological vector space.

Of course, given a diffeological vector space $V$, the projective diffeological vector space $P(V)$ constructed in the proof of the above theorem is functorial but huge. However, there is a natural transformation $P \ra 1$, and $f:V \ra W$ is a linear subduction if and only if $P(f):P(V) \ra P(W)$ is. Here is a non-functorial but much smaller construction, whose proof is similar to the proof of the above theorem:

\begin{prop}
Let $V$ be a diffeological vector space, and let $\{p_i:U_i \ra V\}$ be a generating set of the diffeology on $V$. Then $\sum \eta_i:\oplus F(U_i) \ra V$ is a linear subduction with $\oplus F(U_i)$ a projective diffeological vector space.
\end{prop}

Here is an immediate consequence: 

\begin{cor}\label{cor:projective}
Every projective diffeological vector space is a smooth direct summand of a coproduct of free diffeological vector spaces generated by open subsets of Euclidean spaces.
\end{cor}

\begin{prop}\label{prop:perserveses}
A diffeological vector space $V$ is projective if and only if the functor $L^\infty(V,\blank):\DVect \ra \DVect$ preserves short exact sequences.
\end{prop}
\begin{proof}
($\Leftarrow$) This is clear.

($\Rightarrow$) Since $\DVect$ has enough projectives, $V$ is a retract in $\DVect$ of $P(V)$, which is introduced in the proof of Theorem~\ref{thm:enoughprojectives}. Since the functor $L^\infty(W,\blank)$ is a left adjoint for any diffeological vector space $W$ and subductions are closed under retracts, we are left to show that $L^\infty(P(V),\blank)$ preserves subductions. By Remarks~\ref{rem:naturalisomorphisms} and~\ref{rem:globallift}, we are left to show that for any $U$ open subset of a Euclidean space, the functor $C^\infty(U,\blank):\DVect \ra \DVect$ preserves linear subductions. This then follows from Corollary~\ref{Cor:freeonmfd=>proj}.
\end{proof}

As an easy corollary, we have:

\begin{cor}
If $V$ and $W$ are projective diffeological vector spaces, then so is $V \otimes W$.
\end{cor}

\section{Homological algebra for diffeological vector spaces}\label{s:ha}

In this section, we develop (relative) homological algebra for diffeological vector spaces based on the results we get in the previous sections about linear subductions and projective diffeological vector spaces. We show that every diffeological vector space has a diffeological projective resolution, which is unique up to diffeological chain homotopy equivalence. Shanuel's lemma, Horseshoe lemma and (Short) Five lemma still hold in this setting. From a short exact sequence of diffeological (co)chain complexes, we get a long sequence in $\DVect$ which is also exact in $\Vect$. Then we define ext diffeological vector space $\Ext^n(V,W)$ for any diffeological vector spaces $V,W$ and any $n \in \N$. For any short exact sequence of diffeological vector spaces in the first or second variable, we get a long sequence of ext diffeological vector spaces which is exact in $\Vect$. Finally we show that $\Ext^1(V,W)$ classifies all short exact sequences in $\DVect$ of the form $0 \ra W \ra A \ra V \ra 0$ up to equivalence.

\bigskip

As usual, we define the category $\DCh$ of diffeological chain complexes to be the full subcategory of the functor category $\DVect^{\Z}$ consisting of objects in which the composition of consecutive arrows are $0$, where $\Z$ is viewed as a poset of integers with the opposite ordering. The morphisms in $\DCh$ are called \dftxt{diffeological chain maps}. For any $n \in \Z$, there exists a functor $H_n:\DCh \ra \DVect$ defined by $H_n(\textbf{V})=\ker(d_n)/\im(d_{n+1})$, where 
\[
\xymatrix{V_{n-1} & V_n \ar[l]_-{d_n} & V_{n+1} \ar[l]_-{d_{n+1}}}
\]
is a piece in the diffeological chain complex $\bf{V}$, both $\ker(d_n)$ and $\im(d_{n+1})$ are equipped with the sub-diffeologies of $V_n$, and $H_n(\bf{V})$ is equipped with the quotient diffeology. We call $H_n(\bf{V})$ \dftxt{the $n^{th}$} homology of $\bf{V}$. Two diffeological chain maps $\textbf{f}, \textbf{g}:\textbf{V} \ra \textbf{W}$ are called \dftxt{diffeologically chain homotopic} if there are smooth linear maps $h_n:V_n \ra W_{n+1}$ for all $n \in \Z$ such that $f_n-g_n=h_{n-1} \circ d_n^{\textbf{V}}+d_{n+1}^{\textbf{W}} \circ h_n$ for each $n \in \Z$. This gives an equivalence relation on $\DCh(\textbf{V},\textbf{W})$, called the \dftxt{diffeological chain homotopy equivalence}, which is compatible with compositions in $\DCh$. Write $\hDCh$ for the quotient category. Then the homology functors $H_n$ factors through the projection $\DCh \ra \hDCh$. A diffeological chain map $\bf{f}:\bf{V} \ra \bf{W}$ is called a \dftxt{homology isomorphism} if for each $n \in \Z$, $H_n(\bf{f})$ is an isomorphism in $\DVect$. A diffeological chain complex $\textbf{V}$ is called \dftxt{exact at $n^{th}$ spot} if the induced map $0 \ra V_{n-1}/\ker(d_{n-1}) \ra V_n \ra \im(d_n) \ra 0$ with $V_{n-1}/\ker(d_{n-1})$ equipped with the quotient diffeology of $V_{n-1}$ and $\im(d_n)$ equipped with the sub-diffeology of $V_{n+1}$ is a short exact sequence in $\DVect$. A diffeological chain complex is called \dftxt{exact} if it is exact at every spot. A \dftxt{diffeological projective resolution} of a diffeological vector space $V$ is an exact diffeological chain complex $\bf{V}$ such that $V_{-1}=V$, $V_n=0$ for every $n<-1$, and $V_n$ is projective for every $n \geq 0$. In this case, we write $\textbf{P}(V)$ for the diffeological chain complex with $\textbf{P}(V)_n=V_n$ for each $n \geq 0$ and $\textbf{P}(V)_n=0$ for each $n<0$. Then the diffeological chain map from $\textbf{P}(V)$ to the diffeological chain complex with $V$ concentrated at the $0^{th}$ position is a homology isomorphism. By abuse of notation, we also call this diffeological chain map a \dftxt{diffeological projective resolution}, and denote it by $\textbf{P}(V) \ra V$.

Similarly, one can define the category $\DCh^o$ of diffeological cochain complexes and cohomology functors $H^n:\DCh^o \ra \DVect$.

As the proof of~\cite[Lemma~2.2.5]{We}, one can show that there is a diffeological projective resolution for every diffeological vector space. As the proof of~\cite[Comparison Theorem~2.2.6]{We}, for any diffeological projective resolution $\textbf{V}$ of $V$ (or write as $\textbf{P}(V) \ra V$), any exact diffeological chain complex $\textbf{W}$ with $W_i=0$ for all $i \leq -2$, and any smooth linear map $f:V \ra W_{-1}$, there exists a diffeological chain map $\textbf{V} \ra \textbf{W}$ extending $f$, and the corresponding diffeological chain map $\textbf{P}(V) \ra \bar{\textbf{W}}$ is unique up to diffeological chain homotopy equivalence, where $\bar{\textbf{W}}$ is derived from $\textbf{W}$ by replacing $W_{-1}$ by $0$. In particular, for any two diffeological projective resolutions $\textbf{V}$ and $\textbf{V}'$ of $V$, there is an isomorphism between $\textbf{P}(V)$ and $\textbf{P}'(V)$ in $\hDCh$ from the corresponding diffeological chain maps between $\textbf{V}$ and $\textbf{V}'$ extending $1_V:V \ra V$.

\begin{lem}[Schanuel]\label{lem:Schanuel}
Given short exact sequences of diffeological vector spaces 
\[
\xymatrix{0 \ar[r] & K \ar[r]^i & P \ar[r]^\pi & M \ar[r] & 0}
\]
and
\[
\xymatrix{0 \ar[r] & K' \ar[r]^{i'} & P' \ar[r]^{\pi'} & M \ar[r] & 0}
\]
with both $P$ and $P'$ projective, there is an isomorphism $K \oplus P' \cong K' \oplus P$ in $\DVect$.
\end{lem}
\begin{proof}
By assumption, we have the following commutative diagrams in $\DVect$
\[
\xymatrix{0 \ar[r] & K \ar[r]^i \ar[d]^\alpha & P \ar[r]^\pi \ar[d]^\beta & M \ar[r] \ar[d]^{1_M} & 0 \\ 0 \ar[r] & K' \ar[r]^{i'} \ar[d]^\rho & P' \ar[r]^{\pi'} \ar[d]^\gamma & M \ar[r] \ar[d]^{1_M} & 0 \\ 0 \ar[r] & K \ar[r]^i & P \ar[r]^\pi & M \ar[r] & 0.}
\]
It is shown in the proof of~\cite[Proposition~3.12]{R} that 
\[
\xymatrix{0 \ar[r] & K \ar[r]^-{(i,\alpha)} & P \oplus K' \ar[r]^-{\beta-i'} & P' \ar[r] & 0}
\]
is a short exact sequence in $\Vect$, so we are left to show that this is actually a short exact sequence in $\DVect$. $(i,\alpha)$ is an induction since $i$ is. To show that $(\beta-i')$ is a subduction, note that there exists a smooth linear map $\delta':P' \ra K'$ such that $\beta \circ \gamma - 1_{P'}=i' \delta'$. For any plot $p:U \ra P'$, the map $(\gamma \circ p,\delta' \circ p):U \ra P \oplus K'$ is smooth, and we have $p=(\beta-i') \circ (\gamma \circ p,\delta' \circ p)$.
\end{proof}

\begin{lem}[Horseshoe]\label{lem:Horseshoe}
Let 
\[
\xymatrix{&&& 0 \ar[d] \\ \cdots \ar[r] & P_1' \ar[r]^{d_1'} & P_0' \ar[r]^{\epsilon'} & A' \ar[d]^{i_A} \ar[r] & 0 \\ &&& A \ar[d]^{\pi_A} \\ \cdots \ar[r] & P_1'' \ar[r]^{d_1''} & P_0'' \ar[r]^{\epsilon''} & A'' \ar[r] \ar[d] & 0 \\ &&& 0}
\]
be a diagram in $\DVect$ with the two horizontal lines diffeological projective resolutions of $A'$ and $A''$ respectively, and the vertical line a short exact sequence of diffeological vector spaces. Then there exists a diffeological projective resolution of $A$ in the middle horizontal line making the following diagram commutative and all vertical lines short exact sequences in $\DVect$
\[
\xymatrix{& 0 \ar[d] & 0 \ar[d] & 0 \ar[d] \\ \cdots \ar[r] & P_1' \ar[d]^{i_1} \ar[r]^{d_1'} & P_0' \ar[d]^{i_0} \ar[r]^{\epsilon'} & A' \ar[d]^{i_A} \ar[r] & 0 \\ \cdots \ar[r] & P_1' \oplus P_1''  \ar[d]^{\pi_1} \ar[r]^{d_1} & P_0' \oplus P_0'' \ar[d]^{\pi_0} \ar[r]^{\epsilon} & A \ar[d]^{\pi_A} \ar[r] & 0 \\ \cdots \ar[r] & P_1'' \ar[r]^{d_1''} \ar[d] & P_0'' \ar[d] \ar[r]^{\epsilon''} & A'' \ar[r] \ar[d] & 0 \\ & 0 & 0 & 0}
\]
\end{lem}
\begin{proof}
Since $\pi_A$ is a linear subduction and $P_0''$ is projective, there exists a smooth linear map $f:P_0'' \ra A$ such that $\pi_A \circ f=\epsilon''$. Define $\epsilon := i_A \circ \epsilon' + f$. The inductive proof of~\cite[Horseshoe Lemma~2.2.8]{We} showed that
\[
\xymatrix{& 0 \ar[d] & 0 \ar[d] & 0 \ar[d] \\ 0 \ar[r] & \ker(\epsilon') \ar[r] \ar[d] & P_0' \ar[r]^{\epsilon'} \ar[d] & A' \ar[r] \ar[d]^{i_A} & 0 \\ 0 \ar[r] & \ker(\epsilon) \ar[r] \ar[d] & P_0' \oplus P_0'' \ar[r]^{\epsilon} \ar[d] & A \ar[r] \ar[d]^{\pi_A} & 0 \\ 0 \ar[r] & \ker(\epsilon'') \ar[r] \ar[d] & P_0'' \ar[r]^{\epsilon''} \ar[d] & A'' \ar[r] \ar[d] & 0 \\ & 0 & 0 & 0}
\] 
is a commutative diagram with all vertical and horizontal lines short exact sequences in $\Vect$, so we are left to show that $\epsilon:P_0 :=P_0' \oplus P_0'' \ra A$ is a subduction and $0 \ra \ker(\epsilon') \ra \ker(\epsilon) \ra \ker(\epsilon'') \ra 0$ is a short exact sequence of diffeological vector spaces. Both of these follow from the fact that given plots $p:U \ra P_0''$ and $q:U \ra A$ such that $\epsilon'' \circ p=\pi_A \circ q$, \dftxt{locally} there exists a plot $r:U \ra P_0$ such that $\epsilon \circ r=q$ and $\pi_0 \circ r=p$. Since finite products and coproducts coincide in $\DVect$, this fact follows from the assumptions that $P_0'$ is a diffeological vector space, $\epsilon'$ is a subduction, and $0 \ra A' \ra A \ra A'' \ra 0$ is a short exact sequence in $\DVect$.
\end{proof}

\begin{thm}\label{sestoles}
Let $\textbf{f}:\textbf{A} \ra \textbf{B}$ and $\textbf{g}:\textbf{B} \ra \textbf{C}$ be diffeological chain maps between diffeological chain complexes, such that for each $n$, 
\[
\xymatrix{0 \ar[r] & A_n \ar[r]^{f_n} & B_n \ar[r]^{g_n} & C_n \ar[r] & 0}
\]
is a short exact sequence in $\DVect$. Then we have the following sequence in $\DVect$ which is exact in $\Vect$:
\[
\xymatrix{\cdots \ar[r] & H_n(\textbf{A}) \ar[r]^{f_*} & H_n(\textbf{B}) \ar[r]^{g_*} & H_n(\textbf{C}) \ar@{->} `r/8pt[d] `/10pt[l] `^dl[ll]|{\delta_n} `^r/3pt[dll] [dll] \\ & H_{n-1}(\textbf{A}) \ar[r]^{f_*} & H_{n-1}(\textbf{B}) \ar[r]^{g_*} &  \cdots}
\]
\end{thm}
\begin{proof}
Since every short exact sequence in $\DVect$ is also a short exact sequence in $\Vect$, we have this long exact sequence in $\Vect$. By functoriality of homology functors, we know that all $f_*$'s and $g_*$'s are smooth. So we are left to show that all connecting linear maps $\delta$ are smooth. Recall that $\delta_n:H_n(\textbf{C}) \ra H_{n-1}(\textbf{A})$ is defined as follows:
\[ 
\xymatrix{0 \ar[r] & A_n \ar[r]^{f_n} \ar[d]^{d_n^A} & B_n \ar[r]^{g_n} \ar[d]^{d_n^B} & C_n \ar[r] \ar[d]^{d_n^C} & 0 \\ 0 \ar[r] & A_{n-1} \ar[r]_{f_{n-1}} & B_{n-1} \ar[r]_{g_{n-1}} & C_{n-1} \ar[r] & 0}
\]
Pick a representative $c$ for $[c] \in \ker(d_n^C)/\im (d_{n+1}^C)=H_n(\textbf{C})$. Since $g_n$ is surjective, we can pick $b \in B_n$ such that $g_n(b)=c$. Then there exists $a \in \ker(d_{n-1}^A) \subseteq A_{n-1}$ such that $f_{n-1}(a)=d_n^B(b)$. The map $\delta_n$ is defined by $\delta_n([c])=[a]$, which is independent of the choice of representative $c$ of $[c]$ and lift $b \in B_n$. Any plot $p:U \ra H_n(\textbf{C})$ globally lifts to a plot $q:U \ra \ker(d_n^C)$, since the quotient map $\ker(d_n^C) \ra H_n(\textbf{C})$ is a linear subduction. If we write $i:\ker(d_n^C) \ra C_n$ for the inclusion map, then the plot $i \circ q:U \ra C_n$ globally lifts to a plot $r:U \ra B_n$, since $g_n$ is a linear subduction. So by diagram chasing, the plot $d_n^B \circ r:U \ra B_{n-1}$ globally lifts to a plot $s:U \ra \ker(d_{n-1}^A)$.  This proves the smoothness of $\delta_n$.  
\end{proof}

\begin{de}\label{def:ext}
Given diffeological vector spaces $V$ and $W$, we define 
\[
\Ext^n(V,W):=H^n(L^\infty(\textbf{P}(V),W)),
\] 
where $\textbf{P}(V) \ra V$ is a diffeological projective resolution.
\end{de} 

Here are some basic properties for $\Ext^n(V,W)$:

\begin{rem}\label{rem:propertiesofExt}\
\begin{enumerate} 
\item $\Ext^n(V,W)$ does not depend on the choice of diffeological projective resolution of $V$. 

\item $\Ext^0(V,W)$ is always isomorphic to $L^\infty(V,W)$ in $\DVect$.

\item \label{proj} If $V$ is projective, then $\Ext^n(V,W)=0$ for all $n \geq 1$.

\item We can define \dftxt{injective} diffeological vector space as dual of projective diffeological vector space, i.e., a diffeological vector space $V$ is injective if and only if for every linear induction $f:W_1 \ra W_2$ and every smooth linear map $g:W_1 \ra V$, there exists a smooth linear map $h:W_2 \ra V$ such that $g=h \circ f$. It is direct to show that if $V$ is injective, then $\Ext^n(W,V)=0$ for all $n \geq 1$. In particular, $\Ext^n(W,\R_{ind})=0$ for all $n \geq 1$.

\item Let $\{V_i\}_{i \in I} \cup \{W\}$ be a set of diffeological vector spaces. Then we have natural isomorphisms between $\Ext^n(\oplus_{i \in I}V_i,W)$ and $\prod_{i \in I} \Ext^n(V_i,W)$ in $\DVect$ for all $n \in \N$.

To prove this, besides the usual diagram chasing in homological algebra, we also need the following facts in $\DVect$ (and $\Diff$):
\begin{itemize}
\item Any (co)product of linear subductions is again a linear subduction. For the case of coproduct, it can be proved by the description of the coproduct diffeology from the proof of Proposition~\ref{prop:coproduct}.

\item Let $X_j$ be a subset of a diffeological space $Y_j$ for each $j \in J$. Then the sub-diffeology on $\prod_{j \in J} X_j$ from $\prod_{j \in J} Y_j$ coincides with the product diffeology with each $X_j$ equipped with the sub-diffeology from $Y_j$.

\item Let $A_i$ be a linear subspace of $V_i$ for each $i \in I$. Then we have a natural isomorphism between $(\prod_{i \in I} V_i)/(\prod_{i \in I} A_i)$ and $\prod_{i \in I} (V_i/A_i)$ in $\DVect$.
\end{itemize}

\item Let $\{V\} \cup \{W_j\}_{j \in J}$ be a set of diffeological vector spaces. Then we have natural isomorphisms between $\Ext^n(V,\prod_{j \in J} W_j)$ and $\prod_{j \in J} \Ext^n(V,W_j)$ in $\DVect$ for all $n \in \N$.

\item (Dimension shift) Let 
\[
\xymatrix{V & P_0 \ar[l]_{d_0} & P_1 \ar[l]_{d_1} & \cdots \ar[l] & P_{n-1} \ar[l]_{d_{n-1}} & W \ar[l]_-{d_n}}
\]
be a finite diffeological chain complex such that it is exact at every middle spot, each $P_i$ is projective, $d_0$ is a linear subduction and $d_n$ is a linear induction. Then $\Ext^{m+n}(V,A)$ is isomorphic to $\Ext^m(W,A)$ in $\DVect$ for any $m \in \Z^+$ and any diffeological vector space $A$.
\end{enumerate}
\end{rem}

\begin{thm}\label{thm:les}
Let $\xymatrix{0 \ar[r] & W_1 \ar[r]^i & W_2 \ar[r]^\pi & W_3 \ar[r] & 0}$ be a short exact sequence in $\DVect$. Then for any diffeological vector space $V$, we have the following sequence in $\DVect$ which is exact in $\Vect$:
\[
\xymatrix{0 \ar[r] & L^\infty(V,W_1) \ar[r]^{i_*} & L^\infty(V,W_2) \ar[r]^{\pi_*} & L^\infty(V,W_3) \ar@{->} `r/8pt[d] `/10pt[l] `^dl[ll]|{\delta_0} `^r/3pt[dll] [dll] \\ & \Ext^1(V,W_1) \ar[r]^{i_*} & \Ext^1(V,W_2) \ar[r]^{\pi_*} &  \Ext^1(V,W_3) \ar@{->} `r/8pt[d] `/10pt[l] `^dl[ll]|{\delta_1} `^r/3pt[dll] [dll] \\ & \Ext^2(V,W_1) \ar[r]^{i_*} & \Ext^2(V,W_2) \ar[r]^{\pi_*} & \cdots}
\]
\end{thm}
\begin{proof}
Take a diffeological projective resolution $\textbf{P}(V) \ra V$. Since 
\[
\xymatrix{0 \ar[r] & W_1 \ar[r]^i & W_2 \ar[r]^\pi & W_3 \ar[r] & 0}
\] 
is a short exact sequence in $\DVect$ and each $P_j$ is projective, by Proposition~\ref{prop:perserveses}, we get  diffeological cochain maps 
\[
i_*:L^\infty(\textbf{P}(V),W_1) \ra L^\infty(\textbf{P}(V),W_2)
\] 
and 
\[
\pi_*:L^\infty(\textbf{P}(V),W_2) \ra L^\infty(\textbf{P}(V),W_3)
\] 
between diffeological cochain complexes such that for each $n$ 
\[
\xymatrix{0 \ar[r] & L^\infty(P_n,W_1) \ar[r]^{i_*} & L^\infty(P_n,W_2) \ar[r]^{\pi_*} & L^\infty(P_n,W_3) \ar[r] & 0}
\] 
is a short exact sequence in $\DVect$. The result then follows from the cohomological version of Theorem~\ref{sestoles}.  
\end{proof}

Dually, we have

\begin{thm}\label{thm:les2}
Let $\xymatrix{0 \ar[r] & V_1 \ar[r]^i & V_2 \ar[r]^\pi & V_3 \ar[r] & 0}$ be a short exact sequence in $\DVect$. Then for any diffeological vector space $W$, we have the following sequence in $\DVect$ which is exact in $\Vect$:
\[
\xymatrix{0 \ar[r] & L^\infty(V_3,W) \ar[r]^{\pi^*} & L^\infty(V_2,W) \ar[r]^{i^*} & L^\infty(V_1,W) \ar@{->} `r/8pt[d] `/10pt[l] `^dl[ll]|{\delta^0} `^r/3pt[dll] [dll] \\ & \Ext^1(V_3,W) \ar[r]^{\pi^*} & \Ext^1(V_2,W) \ar[r]^{i^*} &  \Ext^1(V_1,W) \ar@{->} `r/8pt[d] `/10pt[l] `^dl[ll]|{\delta^1} `^r/3pt[dll] [dll] \\ & \Ext^2(V_3,W) \ar[r]^{\pi^*} & \Ext^2(V_2,W) \ar[r]^{i^*} & \cdots}
\]
\end{thm}
\begin{proof}
This follows directly from the Horseshoe lemma and Theorem~\ref{sestoles}.
\end{proof}

\begin{cor}\label{cor:Extandsplit}
Let $V$ and $W$ be diffeological vector spaces. If $\Ext^1(W,V)=0$, then every short exact sequence $0 \ra V \ra A \ra W \ra 0$ in $\DVect$ splits smoothly.
\end{cor}
\begin{proof}
By Theorem~\ref{thm:les}, we know that $L^\infty(W,A) \ra L^\infty(W,W)$ is surjective. The result then follows from Theorem~\ref{thm:splitses}.
\end{proof}

Hence, with the notations in Example~\ref{ex:nonssplit}, $\Ext^1(\prod_\omega \R,K) \neq 0$.

\begin{cor}\label{cor:ext1=>projective}
A diffeological vector space $V$ is projective if and only if 
\[
\Ext^1(V,W)=0
\] 
for every diffeological vector space $W$.
\end{cor}
\begin{proof}
($\Rightarrow$) This follows from Remark~\ref{rem:propertiesofExt}(\ref{proj}).

($\Leftarrow$) This follows from Corollary~\ref{cor:Extandsplit} and Remark~\ref{rem:projective}(\ref{projiffsplit}).
\end{proof}

As an easy corollary, we know that if $V \ra W$ is a linear subduction between projective diffeological vector spaces, then its kernel is also projective. In particular, if $f:M \ra N$ is a smooth fiber bundle between smooth manifolds, then the kernel of $F(f):F(M) \ra F(N)$ is a projective diffeological vector space. 

\bigskip

Finally we are going to prove the inverse of Corollary~\ref{cor:Extandsplit}. 

For $V,W$ diffeological vector spaces, we define two short exact sequences $0 \ra V \ra A \ra W \ra 0$ and $0 \ra V \ra A' \ra W \ra 0$ in $\DVect$ to be equivalent if there is a commutative diagram in $\DVect$ 
\[
\xymatrix{0 \ar[r] & V \ar[r] \ar[d]^{1_V} & A \ar[r] \ar[d]^f & W \ar[r] \ar[d]^{1_W} & 0 \\ 0 \ar[r] & V \ar[r] & A' \ar[r] & W \ar[r] & 0}
\]
with $f$ an isomorphism in $\DVect$. (We will see in the next lemma that indeed only the existence of $f$ is needed.) This defines an equivalence relation on the set of all short exact sequences in $\DVect$ of the form $0 \ra V \ra A \ra W \ra 0$, and we write the quotient set as $e(W,V)$.

\begin{lem}\label{lem:sfl}
Given a commutative diagram in $\DVect$
\[
\xymatrix{0 \ar[r] & V \ar[r]^i \ar[d]^g & A \ar[r]^\pi \ar[d]^f & W \ar[r] \ar[d]^h & 0 \\ 0 \ar[r] & V' \ar[r]_{i'} & A' \ar[r]_{\pi'} & W' \ar[r] & 0}
\]
with both rows short exact sequences in $\DVect$, and both $g$ and $h$ isomorphisms in $\DVect$. Then $f$ is also an isomorphism in $\DVect$.
\end{lem}
\begin{proof}
From~\cite[The Short Five Lemma~1.17]{H}, we know that $f$ is an isomorphism in $\Vect$. We are left to show that $f^{-1}$ is smooth. For any plot $p:U \ra A'$, since $h^{-1} \circ \pi' \circ p:U \ra W$ is smooth and $\pi$ is a linear subduction, we have a smooth linear map $q:U \ra A$ such that $\pi \circ q=h^{-1} \circ \pi' \circ p$. Then $\pi \circ (f^{-1} \circ p-q)=0$ implies that $\pi' \circ (p-f \circ q)=0$. Since the bottom row in the above diagram is short exact in $\DVect$, there exists a smooth linear map $r:U \ra V'$ such that $f \circ q - p=i' \circ r$. So $f^{-1} \circ p=q+f^{-1} \circ i' \circ r=q+i \circ g^{-1} \circ r$ is smooth, which implies that $f^{-1}$ is smooth.
\end{proof}

Slightly more generally, we can prove the following in a similar way:

\begin{prop}\label{prop:fl}
Let
\[
\xymatrix{V_1 \ar[r] \ar[d]^{f_1} & V_2 \ar[r] \ar[d]^{f_2} & V_3 \ar[r]^{d_3} \ar[d]^{f_3} & V_4 \ar[r] \ar[d]^{f_4} & V_5 \ar[d]^{f_5} \\ W_1 \ar[r] & W_2 \ar[r]_{\partial_2} & W_3 \ar[r] & W_4 \ar[r] & W_5}
\]
be a commutative diagram in $\DVect$, such that both rows are exact in $\Vect$, $f_1$ is surjective, $f_5$ is injective, and both $f_2$ and $f_4$ are isomorphisms in $\DVect$. If the induced maps $V_3 \ra \im(d_3)$ is a subduction and $W_2/\ker(\partial_2) \ra W_3$ is an induction, then $f_3$ is an isomorphism in $\DVect$.
\end{prop}

\begin{thm}\label{thm:bijection}
There is a bijection between $e(W,V)$ and $\Ext^1(W,V)$ which sends the class of smoothly split short exact sequences in $e(W,V)$ to $0$ in $\Ext^1(W,V)$.
\end{thm}  
\begin{proof}
As explained in~\cite[pp.~422-423]{R}, the map $\psi:e(W,V) \ra \Ext^1(W,V)$ is defined as follows: For any short exact sequence $0 \ra V \ra A \ra W \ra 0$ in $\DVect$ and any diffeological projective resolution $\textbf{P}(W) \ra W$, we have a commutative diagram (not necessarily unique) in $\DVect$
\[
\xymatrix{\cdots \ar[r] & P_2 \ar[r] \ar[d] & P_1 \ar[r] \ar[d]^f & P_0 \ar[r] \ar[d] & W \ar[r] \ar[d]^{1_W} & 0 \\ & 0 \ar[r] & V \ar[r] & A \ar[r] & W \ar[r] & 0.}
\]
Then $\psi$ takes the class of $0 \ra V \ra A \ra W \ra 0$ in $e(W,V)$ to $[f] \in \Ext^1(W,V)$. As shown in~\cite{R} that this map depends neither on representing short exact sequence in $e(W,V)$ nor on diffeological projective resolution of $W$ nor on the choice of $f$. The second statement of the theorem can then be proved as~\cite[Lemma~7.27]{R}.

The inverse map $\theta:\Ext^1(W,V) \ra e(W,V)$ can be constructed as follows: For any diffeological projective resolution $\textbf{P}(W) \ra W$ and any $f \in L^\infty(P_1,V)$ with $d_2^*(f)=0$, let $A$ be the pushout of $\xymatrix{V & P_1 \ar[r]^{d_1} \ar[l]_f & P_0}$ in $\DVect$. As proved in~\cite[Lemma~7.28]{R} that we have a short exact sequence 
\[
\xymatrix{0 \ar[r] & V \ar[r]^i & A \ar[r]^\pi & W \ar[r] & 0}
\] 
in $\Vect$. It is easy to check that $\pi$ is a subduction since $P_0 \ra W$ is, and $i$ is an induction since $A$ is a pushout, $\ker(d_0) \ra P_0$ is an induction, and $P_1 \ra \ker(d_0)$ is a subduction. As shown in~\cite[Theorem~7.30]{R} that $\theta$ is independent of the choice of representative of $[f] \in \Ext^1(W,V)$, and $\theta$ is the inverse of $\psi$. 
\end{proof}

As an immediate corollary to Theorem~\ref{thm:bijection} and Corollary~\ref{cor:Extandsplit}, we know that for any fixed diffeological vector spaces $V$ and $W$, $\Ext^1(W,V)=0$ if and only if every short exact sequence $0 \ra V \ra A \ra W \ra 0$ in $\DVect$ splits smoothly.

\end{document}